\documentclass[11pt]{article}
\usepackage{geometry,graphicx,latexsym,amssymb,verbatim,amsthm}
\usepackage{tikz}
\usetikzlibrary{calc}
\usepackage{hyperref}
\usepackage{todonotes}

\newtheorem{thm}{Theorem}
\newtheorem{ob}[thm]{Observation}

\newtheorem{cor}[thm]{Corollary}

\newtheorem{conj}[thm]{Conjecture}
\newtheorem{quest}[thm]{Question}
\newtheorem{clm}{Claim}[thm]

\newcommand{\barD}{\overline{D}}
\newcommand{\barS}{\overline{S}}
\newcommand{\cP}{\mathcal{P}}

\newcommand{\epn}{{\rm epn}}
\newcommand{\diam}{{\rm diam}}

\newcommand{\gt}{\gamma_t}
\newcommand{\gL}{\gamma_L}
\newcommand{\gLT}{\gamma_t^L}

\newcommand{\smallqed}{{\tiny ($\Box$)}}

\newenvironment{unnumbered}[1]{\trivlist
\item [\hskip \labelsep {\bf #1}]\ignorespaces\it}{\endtrivlist}

\newcommand{\claimproof}{\noindent\emph{Proof of claim.} }

\begin{document}

\title{Locating-Total Dominating Sets in Twin-Free Graphs: a Conjecture}

\author{$^{1,2}$Florent Foucaud and $^1$Michael A. Henning\\
\\
$^1$Department of Pure and Applied Mathematics \\
University of Johannesburg \\
Auckland Park, 2006 South Africa\\
E-mail: mahenning@uj.ac.za \\
\\
$^2$LIMOS\\
Universit\'e Blaise Pascal\\
Clermont-Ferrand, France\\
E-mail: florent.foucaud@gmail.com \\
}

\maketitle

\begin{abstract}
A total dominating set of a graph $G$ is a set $D$ of vertices of $G$ such that every vertex of $G$ has a neighbor in $D$. A locating-total dominating set of $G$ is a total dominating set $D$ of $G$ with the additional property that every two distinct vertices outside $D$ have distinct neighbors in $D$; that is, for distinct vertices $u$ and $v$ outside $D$, $N(u) \cap D \ne N(v) \cap D$ where $N(u)$ denotes the open neighborhood of $u$. A graph is twin-free if every two distinct vertices have distinct open and closed neighborhoods.  The location-total domination number of $G$, denoted $\gLT(G)$, is the minimum cardinality of a locating-total dominating set in $G$. It is well-known that every connected graph of order $n \ge 3$ has a total dominating set of size at most $\frac{2}{3}n$. We conjecture that if $G$ is a twin-free graph of order~$n$ with no isolated vertex, then $\gLT(G) \le \frac{2}{3}n$. We prove the conjecture for graphs without $4$-cycles as a subgraph. We also prove that if $G$ is a twin-free graph of order~$n$, then $\gLT(G) \le \frac{3}{4}n$.
\end{abstract}


\section{Introduction}

A \emph{dominating set} in a graph $G$ is a set $D$ of vertices of $G$ such that every vertex outside $D$ is adjacent to a vertex in $D$. The \emph{domination number}, $\gamma(G)$, of $G$ is the minimum cardinality of a dominating set in $G$. A \emph{total dominating set}, abbreviated TD-set, of $G$ is a set $D$ of vertices of $G$ such that every vertex of $G$ is adjacent to a vertex in $D$. The \emph{total domination number} of $G$, denoted by $\gt(G)$, is the minimum cardinality of a TD-set in $G$.  The literature on the subject of domination parameters in graphs up to the year 1997 has been surveyed and detailed in the two books~\cite{hhs1, hhs2}, and a recent book on total dominating sets is also available~\cite{bookTD}.

Among the existing variations of (total) domination, the ones of \emph{location-domination} and \emph{location-total domination} are widely studied. A set $D$ of vertices \emph{locates} a vertex $v$ if the neighborhood of $v$ within $D$ is unique among all vertices in $V(G)\setminus D$. A \emph{locating}-\emph{dominating set} is a dominating set $D$ that locates all the vertices, and the \emph{location-domination number} of $G$, denoted $\gL(G)$, is the minimum cardinality of a locating-dominating set in $G$. A \emph{locating}-\emph{total dominating set}, abbreviated LTD-set, is a TD-set $D$ that locates all the vertices, and the \emph{location-total domination number} of $G$, denoted $\gLT(G)$, is the minimum cardinality of a LTD-set in $G$. The concept of a locating-dominating set was introduced and first studied by Slater~\cite{s2,s3} (see also~\cite{cst,fh,Heia,rs,s4}), and the additional condition that the locating-dominating set be a total dominating set  was first considered in~\cite{hhh06} (see also~\cite{BCMMS07,BD11,BFL08,C08,CR09,CS11,hl12,hr12}).

We remark that there are (twin-free) graphs with total domination number two and arbitrarily large location-total domination number. For $k \ge 3$, let $G_k$ be the graph obtained from $K_{2,k}$ as follows: select one of the two vertices of degree~$k$ and subdivide every edge incident with it; then, add an edge joining the two vertices of degree~$k$; finally, add two new vertices of degree~$1$, each adjacent to one of the degree~$k$-vertices. The resulting graph, $G_k$, has order $2k+4$, total domination number~$2$, and we claim that its location-total domination number is exactly one-half the order (namely, $k+2$). One possible LTD-set of $G_k$ consists of the two vertices of degree~$k+1$, and for each pair of adjacent vertices of degree~$2$, one of the vertices of that pair belongs to the LTD-set. The graph $G_4$, for example, is illustrated in Figure~\ref{f:G4}, where the darkened vertices form an LTD-set in $G_4$. To see that no smaller LTD-set exists, observe first that the two vertices of degree~$k+1$ must belong to any LTD-set of $G_k$ (otherwise, the two vertices of degree~$1$ are not totally dominated). Moreover, consider any set of two pairs of adjacent vertices of degree~$2$ in $G_k$. In order for these four vertices to be located, at least one of them must belong to any LTD-set (otherwise, the ones adjacent to the same vertex of degree~$k+1$ are not located). This shows that for at least $k-1$ pairs of adjacent degree~$2$-vertices, one member of that pair belongs to any LTD-set of $G_k$. Thus, any LTD-set of $G_k$ has size at least $k+1$. Assuming that we have an LTD-set of size exactly $k+1$, then we have a pair of adjacent vertices of degree~$2$ not belonging to the LTD-set, and moreover none of the degree~$1$-vertices belongs to the LTD-set. But then each of the two above degree~$2$-vertices and each degree~$1$-vertex of $G_k$ is totally dominated only by its neighbor of degree~$k+1$ and is therefore not located, a contradiction. Hence $\gLT(G_k)=k+2$, as claimed.

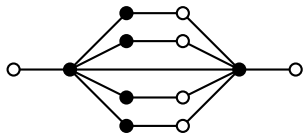
\begin{figure}[htb]
\tikzstyle{every node}=[circle, draw, fill=black!0, inner sep=0pt,minimum width=.16cm]
\begin{center}
\begin{tikzpicture}[thick,scale=.6]
  \draw(0,0) { 
    +(2.50,2.50) -- +(3.75,2.50)
    +(3.75,1.88) -- +(2.50,1.88)
    +(2.50,2.50) -- +(1.25,1.25)
    +(1.25,1.25) -- +(2.50,1.88)
    +(1.25,1.25) -- +(2.50,0.63)
    +(1.25,1.25) -- +(2.50,0.00)
    +(2.50,0.00) -- +(3.75,0.00)
    +(3.75,0.00) -- +(5.00,1.25)
    +(5.00,1.25) -- +(3.75,1.88)
    +(3.75,2.50) -- +(5.00,1.25)
    +(5.00,1.25) -- +(3.75,0.63)
    +(3.75,0.63) -- +(2.50,0.63)
    +(1.25,1.25) -- +(5.00,1.25)
    +(6.25,1.25) -- +(5.00,1.25)
    +(0,1.25) -- +(1.25,1.25)
    +(3.75,0.63) node{}
    +(2.50,0.63) node[fill=black!100]{}
    +(2.50,1.88) node[fill=black!100]{}
    +(2.50,2.50) node[fill=black!100]{}
    +(3.75,2.50) node{}
    +(3.75,1.88) node{}
    +(2.50,0.00) node[fill=black!100]{}
    +(3.75,0.00) node{}
    +(5.00,1.25) node[fill=black!100]{}
    +(1.25,1.25) node[fill=black!100]{}
    +(0,1.25) node{}
    +(6.25,1.25) node{}

  };
\end{tikzpicture}
\end{center}
\vskip -0.6 cm \caption{The twin-free graph $G_4$.}
\label{f:G4}
\end{figure}

A classic result due to Cockayne et al.~\cite{CDH80} states that every connected graph of order at least~$3$ has a TD-set of cardinality at most two-thirds its order. While there are many graphs (without isolated vertices) which have location-total domination number much larger than two-thirds their order, the only such graphs that are known contain many \emph{twins}, that is, pairs of vertices with the same closed or open neighborhood. We conjecture that in fact, twin-free graphs have location-total domination number at most two-thirds their order. In this paper we initiate the study of this conjecture.

\noindent\textbf{Definitions and notations.} For notation and graph theory terminology, we in general follow~\cite{hhs1}. Specifically, let $G$ be a graph with vertex set $V(G)$, edge set $E(G)$ and with no isolated vertex. The \emph{open neighborhood} of a vertex $v \in V(G)$ is $N_G(v) = \{u \in V \, | \, uv \in E(G)\}$ and its \emph{closed neighborhood} is the set $N_G[v] = N_G(v) \cup \{v\}$. The \emph{degree} of $v$ is $d_G(v) = |N_G(v)|$.
For a set $S \subseteq V(G)$, its \emph{open neighborhood} is the set $N_G(S) = \bigcup_{v \in S} N_G(v)$, and its \emph{closed neighborhood} is the set $N_G[S] = N_G(S) \cup S$.
Given a set $S \subset V(G)$ and a vertex $v \in S$, an \emph{$S$-external private neighbor} of $v$ is a vertex outside $S$ that is adjacent to $v$ but to no other vertex of $S$ in $G$. The set of all $S$-external private neighbors of $v$, abbreviated $\epn_G(v,S)$, is the \emph{$S$-external private neighborhood}.
The subgraph induced by a set $S$ of vertices in $G$ is denoted by $G[S]$.
If the graph $G$ is clear from the context, we simply write $V$, $E$, $N(v)$, $N[v]$, $N(S)$, $N[S]$, $d(v)$ and $\epn(v,S)$ rather than $V(G)$, $E(G)$, $N_G(v)$, $N_G[v]$, $N_G(S)$, $N_G[S]$, $d_G(v)$ and $\epn_G(v,S)$, respectively.
%

Given a set $S$ of edges in $G$, we will denote by $G-S$ the subgraph obtained from $G$ by deleting all edges of $S$. For a set $S$ of vertices, $G-S$ is the graph obtained from $G$ by removing all vertices of $S$ and removing all edges incident to vertices of $S$. A \emph{cycle} on $n$ vertices is denoted by $C_n$ and a \emph{path} on $n$ vertices by $P_n$. The \emph{girth} of $G$ is the length of a shortest cycle in $G$.

A set $D$ is a dominating set of $G$ if $N[v] \cap D \ne \emptyset$ for every vertex $v$ in $G$, or, equivalently, $N[D] = V(G)$. A set $D$ is a TD-set of $G$ if $N(v) \cap D \ne \emptyset$ for every vertex $v$ in $G$, or, equivalently, $N(D) = V(G)$.
Two distinct vertices $u$ and $v$ in $V(G) \setminus D$ are \emph{located} by $D$ if they have distinct neighbors in $D$; that is, $N(u) \cap D \ne N(v) \cap D$. If a vertex $u \in V(G) \setminus D$ is located from every other vertex in $V(G)\setminus D$, we simply say that $u$ is \emph{located} by $D$.

A set $S$ is a \emph{locating set} of $G$ if every two distinct vertices outside $S$ are located by $S$. In particular, if $S$ is both a dominating set and a locating set, then $S$ is a locating-dominating set. Further, if $S$ is both a TD-set and a locating set, then $S$ is a \emph{locating-total dominating set}. We remark that the only difference between a locating set and a locating-dominating set in $G$ is that a locating set might have a unique non-dominated vertex.

%

Two distinct vertices $u$ and $v$ of a graph $G$ are \emph{open twins} if $N(u)=N(v)$ and \emph{closed twins} if $N[u]=N[v]$. Further, $u$ and $v$ are \emph{twins} in $G$ if they are open twins or closed twins in $G$. A graph is \emph{twin-free} if it has no twins.

For two vertices $u$ and $v$ in a connected graph $G$, the \emph{distance} $d_G(u,v)$ between $u$ and $v$ is the length of a shortest $(u,v)$-path in $G$. The maximum distance among all pairs of vertices of $G$ is the \emph{diameter} of $G$, which is denoted by $\diam(G)$. A \emph{nontrivial connected graph} is a connected graph of order at least~$2$. A \emph{leaf} of graph $G$ is a vertex of degree~$1$, while a \emph{support vertex} of $G$ is a vertex adjacent to a leaf.

A \emph{rooted tree} $T$ distinguishes one vertex $r$ called the \emph{root}. For each vertex $v \ne r$ of $T$, the \emph{parent} of $v$ is the neighbor of $v$ on the unique $(r,v)$-path, while a \emph{child} of $v$ is any other neighbor of $v$. A \emph{descendant} of $v$ is a vertex $u \ne v$ such that the unique $(r,u)$-path contains $v$. Thus, every child of $v$ is a descendant of $v$. We let $D(v)$ denote the set of descendants of $v$, and we define $D[v] = D(v) \cup \{v\}$. The \emph{maximal subtree} at $v$ is the subtree of $T$ induced by $D[v]$, and is denoted by $T_v$.

The \emph{$2$-corona} of a graph $H$ is the graph of order~$3|V(H)|$ obtained from $H$ by adding a vertex-disjoint copy of a path $P_2$ for each vertex $v$ of $H$ and adding an edge joining $v$ to one end of the added path.

We use the standard notation $[k] = \{1,2,\ldots,k\}$. If $A$ and $B$ are sets, then $A \times B = \{(a,b) \mid a \in A, b \in B\}$.


\noindent\textbf{Conjectures and known results.} As a motivation for our study, we pose and state the following conjecture.

\begin{conj}\label{conj}
Every twin-free graph $G$ of order $n$ without isolated vertices satisfies $\gLT(G)\le \frac{2}{3}n$.
\end{conj}

In an earlier paper, Henning and L\"owenstein~\cite{hl12} proved that every connected cubic claw-free graph (not necessarily twin-free) has a LTD-set of size at most one-half its order, which implies that Conjecture~\ref{conj} is true for such graphs. Moreover they conjectured this to be true for every connected cubic graph, with two exceptions --- which, if true, would imply Conjecture~\ref{conj} for all cubic graphs.

A similar conjecture for locating-dominating sets, that motivated the present study,  was posed in~\cite{conjpaper}, and was strengthened in~\cite{Heia}.\footnote{Note that in~\cite{Heia}, we mistakenly attributed Conjecture~\ref{conj-LD} to the authors of~\cite{conjpaper}. We discuss this in more detail in~\cite{cubic}.}

\begin{conj}[Garijo, Gonz\'alez, M\'arquez~\cite{conjpaper}]\label{conj-LD-original}
There exists an integer $n_1$ such that for any $n\geq n_1$, the maximum value of the location-domination number of a connected twin-free graph of order~$n$ is $\lfloor\frac{n}{2}\rfloor$.
\end{conj}

\begin{conj}[Foucaud, Henning, L\"owenstein, Sasse~\cite{cubic,Heia}]\label{conj-LD}
Every twin-free graph $G$ of order~$n$ without isolated vertices satisfies $\gL(G)\le \frac{n}{2}$.
\end{conj}

Conjecture~\ref{conj-LD} remains open, although it was proved for a number of graph classes such as bipartite graphs and graphs with no $4$-cycles~\cite{conjpaper}, split and co-bipartite graphs~\cite{Heia}, and cubic graphs~\cite{cubic}. Some of these results were obtained using selected vertex covers and matchings, but none of these techniques seems to be useful in the study of Conjecture~\ref{conj}.

\noindent\textbf{Our results.} We prove the bound $\gLT(G)\le \frac{3}{4}n$ in Section~\ref{sec:general}. We then  give support to Conjecture~\ref{conj} by proving it for graphs without $4$-cycles in Section~\ref{sec:noC4}, where we also characterize all extremal examples without $4$-cycles. (In this paper, by ``graph with no $4$-cycles'', we mean that the graph does not contain any $4$-cycle as a subgraph, whether the $4$-cycle is induced or not.) We also discuss Conjecture~\ref{conj} in relation with the minimum degree in Section~\ref{sec:mindegree}, and we conclude the paper in Section~\ref{sec:conclu}.

\section{Preliminaries}

This section contains a number of preliminary results that will be useful in the next sections.

\begin{thm}[Cockayne et al.~\cite{CDH80}; Brigham et al.~\cite{BCV00}] \label{thm:TD-extremal}
If $G$ is a connected graph of order $n \ge 3$, then $\gt(G)\le \frac{2}{3}n$. Further, $\gt(G) = \frac{2}{3}n$ if and only if $G$ is isomorphic to a $3$-cycle, a $6$-cycle, or the $2$-corona of some connected graph $H$.
\end{thm}

We will need the following property of minimum TD-sets in a graph established in~\cite{H00}.

\begin{thm}[\cite{H00}]\label{t:BC}
If $G$ is a connected graph of order~$n \ge 3$, and $G \not\cong K_n$, then $G$ has a minimum TD-set $S$ such that every vertex $v \in
S$ satisfies $|\epn(v,S)| \ge 1$ or has a neighbor~$x$ in $S$ of degree~$1$ in $G[S]$ satisfying $|\epn(x,S)| \ge 1$.
\end{thm}

Given a graph $G$, the set $L \cup T$, where $L$ is a locating-dominating set of $G$, and $T$ is a TD-set of $G$ is both a TD-set and a locating set, implying the following observation.

\begin{ob}\label{ob:D=T+L}
For every graph $G$ without isolated vertices, we have $\gLT(G) \le \gL(G)+\gamma_t(G)$.
\end{ob}

\section{A general upper bound of three-quarters the order}
\label{sec:general}

In this section we prove a general upper bound on the location-total domination number of a graph in terms of its order. The proof is similar to the bound $\gL(G)\le \frac{2}{3}n$ proved for locating-dominating sets in~\cite{Heia}.

\begin{thm}
If $G$ is a twin-free graph of order~$n$ without isolated vertices, then $\gLT(G) \le \frac{3}{4}n$.
 \label{t:generalbd}
\end{thm}
\begin{proof} By linearity, we may assume that $G$ is connected. By the twin-freeness of $G$, we note that $n \ge 4$ and that $G \not\cong K_n$. For an arbitrary subset $S$ of vertices in $G$, let $\cP_S$ be a partition of $\barS = V(G) \setminus S$ with the property that all vertices in the same part of the partition have the same open neighborhood in $S$ and vertices from different parts of the partition have different open neighborhood in $S$. Let $|\cP_S| = k(S)$. Let $X_S$ be the set of vertices in $\barS$ that belong to a partition set in $\cP_S$ of size~$1$ and let $Y_S = \barS \setminus X_S$. Hence every vertex in $Y_S$ belongs to a partition set of size at least~$2$. Let $n_1(S) = |X_S|$ and let $n_2(S) = k(S) - n_1(S)$.
Let $S$ be a minimum TD-set in $G$ with the property that every vertex $v \in S$ satisfies $|\epn(v,S)| \ge 1$ or has a neighbor $v'$ in $S$ of degree~$1$ in $G[S]$ satisfying $|\epn(v',S)| \ge 1$. Such a set exists by Theorem~\ref{t:BC}. We note that at least half the vertices in $S$ have an $S$-external private neighbor, implying that $n_1(S) + n_2(S) \ge \frac{1}{2}|S|$.
Among all supersets $S'$ of $S$ with the property that $n_1(S') + n_2(S') \ge \frac{1}{2}|S'|$, let $D$ be chosen to be inclusion-wise maximal. (Possibly, $D = S$.)

\begin{clm}\label{claimA}
The vertices in each partition set of size at least~$2$ in $\cP_D$ have distinct neighborhoods in $X_D$, and $D \cup X_D$ is a LTD-set of $G$.
\end{clm}
\claimproof
Let $u$ and $v$ be two vertices that belong to a partition set $T$, of size at least~$2$ in $\cP_D$. Since $G$ is twin-free, there exists a vertex $w \notin \{u,v\}$ that is adjacent to exactly one of $u$ and~$v$.
Since $u$ and $v$ have the same neighbors in $D$, we note that $w \notin D$.
Hence, $w \in \barD = V(G) \setminus D$.
Suppose that $w \in Y_D$ and consider the set $D' = D \cup \{w\}$.
Let $R$ be an arbitrary partition set in $\cP_{D}$ that might or might not contain $w$.
If $w$ is either adjacent to every vertex of $R\setminus\{w\}$ or adjacent to no vertex in $R\setminus\{w\}$, then $R\setminus\{w\}$ is a partition set in $\cP_{D'}$.
If $w$ is adjacent to some, but not all, vertices of $R\setminus\{w\}$, then there is a partition $R\setminus\{w\} = (R_1,R_2)$ of $R\setminus\{w\}$ where $R_1$ are the vertices in $R\setminus\{w\}$ adjacent to $w$ and $R_2$ are the remaining vertices in $R\setminus\{w\}$.
In this case, both sets $R_1$ and $R_2$ form a partition set in $\cP_{D'}$.
In particular, we note that there is a partition $T\setminus\{w\} = (T_1,T_2)$ of $T\setminus\{w\}$ where both sets $T_1$ and $T_2$ form a partition set in $\cP_{D'}$.
Therefore, $n_1(D') + n_2(D') \ge n_1(D) + n_2(D) + 1 \ge \frac{1}{2}|D| + 1 > \frac{1}{2}(|D| + 1) = \frac{1}{2}|D'|$, contradicting the maximality of $D$. Hence, $w \notin Y_D$. Therefore, $w \in X_D$.
Hence, $u$ and $v$ are located by the set $X_D$ in $G$.
Moreover, $D\cup X_D$ is a TD-set since $D$ itself is a TD-set.~\smallqed

\medskip
Let $Y'_D$ be obtained from $Y_D$ by deleting one vertex from each partition set of size at least~$2$ in $\cP_D$, and let $D' = D \cup Y'_D$. Then, $|D'| = n - n_1(D) - n_2(D)$. By definition of the partition $\cP_D$, every vertex in $V(G) \setminus D'$ has a distinct nonempty neighborhood in $D$ and therefore in $D'$. Moreover, $D'$ is a TD-set since $D$ itself is a TD-set. Hence we have the following claim.

\begin{clm}\label{claimB}
The set $D'$ is a LTD-set of $G$.
\end{clm}

Let $n_1 = n_1(D)$ and $n_2 = n_2(D)$. By Claim~\ref{claimA}, the set $D \cup X_D$ is a LTD-set of $G$ of cardinality~$|D| + n_1$. By Claim~\ref{claimB}, the set $D'$ is a LTD-set of $G$ of cardinality~$n - n_1 - n_2$. Hence,
\begin{equation}
\gLT(G) \le \min \{ |D| + n_1, n - n_1 - n_2 \}.
\label{Eq1}
\end{equation}

Inequality~(\ref{Eq1}) implies that if $n - n_1 - n_2 \le \frac{3}{4}n$, then $\gL(G) \le \frac{3}{4}n$. Hence we may assume that $n - n_1 - n_2 > \frac{3}{4}n$, for otherwise the desired upper bound on $\gLT(G)$ follows. With this assumption, $n_1 + n_2 < \frac{1}{4}n$. By our choice of the set $D$, we recall that $|D| \le 2(n_1 + n_2)$. Therefore,
\[
|D| + n_1 \le 3n_1 + 2n_2 \le 3(n_1 + n_2) < \frac{3}{4}n.
\]

Hence, by Inequality~(\ref{Eq1}), $\gLT(G) < \frac{3}{4}n$. This completes the proof of Theorem~\ref{t:generalbd}.
\end{proof}

\section{Graphs without $4$-cycles}\label{sec:noC4}

In this section, we prove Conjecture~\ref{conj} for graphs with no $4$-cycles. We also characterize all graphs with no $4$-cycles that achieve the bound of Conjecture~\ref{conj}. Surprisingly, these are precisely those graphs that have no $4$-cycles and no twins and that are extremal for the bound on the total domination number from Theorem~\ref{thm:TD-extremal}. This is in stark contrast with Conjecture~\ref{conj-LD} for the location-domination number, where many graphs (without $4$-cycles) are known that are extremal for the conjecture but have much smaller domination number than one-half the order, see~\cite{Heia}.

\begin{thm}
Let $G$ be a twin-free graph of order $n$ without isolated vertices and $4$-cycles. Then, $\gLT(G) \le \frac{2}{3}n$. Further, $\gLT(G) = \frac{2}{3}n$ if and only if $G$ is isomorphic to a $6$-cycle or is the $2$-corona of some connected nontrivial graph that contains no $4$-cycles.
\label{t:no4cycle}
\end{thm}
\begin{proof}
We prove the theorem by induction on $n$. By linearity, we may assume that $G$ is connected, for otherwise we apply induction to each component of $G$ and we are done. By the twin-freeness of $G$, we note that $n \ge 4$. Further if $n = 4$, then since $G$ is $C_4$-free, the graph $G$ is the path $P_4$ and $\gLT(P_4)=2 < \frac{2}{3}n$. This establishes the base case. Let $n \ge 5$ and assume that every twin-free graph $G'$ without isolated vertices and with no $4$-cycles of order~$n'$, where $n' < n$, satisfies $\gLT(G') \le \frac{2}{3}n'$, and that the only graphs achieving the bound are the extremal graphs described in Theorem~\ref{thm:TD-extremal} that are twin-free and have no $4$-cycles. Let $G$ be a twin-free graph without isolated vertices and with no $4$-cycles of order~$n$. The general idea will be to partition $V(G)$ into two sets $V_1$ and $V_2$. If $G[V_1]$ and/or $G[V_2]$ are twin-free, we apply induction, and use the obtained LTD-sets of $G[V_1]$ and/or $G[V_2]$ to build one of $G$. We proceed further with the following series of claims.

\medskip
\begin{clm}\label{claimtree}
If $G$ is a tree, then $\gLT(G)\le \frac{2}{3}n$. Further, $\gLT(G) = \frac{2}{3}n$ if and only if $G$ is the $2$-corona of a nontrivial tree.
\end{clm}
\noindent
\emph{Proof of Claim~\ref{claimtree}.}
Suppose that $G$ is a tree. Since $n \ge 5$, we note that $\diam(G) \ge 4$ (otherwise $G$ contains twin vertices of degree~$1$). For the same reason, if $\diam(G) = 4$, then either $G = P_5$ or $G$ is obtained from a star $K_{1,k+1}$, where $k \ge 2$, by subdividing at least~$k$ edges of the star exactly once. In this case, the set of vertices of degree at least~$2$ in $G$ forms a LTD-set of size strictly less than two-thirds the order. Hence, we may assume that $\diam(G) \ge 5$, for otherwise the desired result follows.

Let $P$ be a longest path in $G$ and let $P$ be an $(r,u)$-path. Necessarily, both $r$ and $u$ are leaves. Since $\diam(G) \ge 5$, we note that $P$ has length at least~$5$. We now root the tree at the vertex $r$. Let $v$ be the parent of $u$, and let $w$ be the parent of $v$, $x$ the parent of $w$, and $y$ the parent of $x$ in the rooted tree. Since $|V(P)| \ge 6$, we note that $y \ne r$. Since $G$ is twin-free, the vertex $w$ has at most one leaf-neighbor and every child of $w$ that is not a leaf has degree~$2$ in $G$. In particular, $d_G(v)=2$.
We now consider the subtree $G_w$ of $G$ rooted at the vertex~$w$. If $d_G(w) = 2$, then $G_w = P_3$, while if $d_G(w) \ge 3$, then $G_w$ is obtained from a star $K_{1,k+1}$, where $k \ge 1$, by subdividing at least~$k$ edges of the star exactly once. Let $G' = G - V(G_w)$.

We now define the subtrees $G_1$ and $G_2$ of $G$ as follows. We distinguish two cases; in both of them, $G_2$ is twin-free.

\begin{itemize}
\item If the tree $G'$ is twin-free, then we let $V_1 = V(G_w)$ and $V_2 = V(G) \setminus V_1$, and we let $G_1 = G[V_1]$ and $G_2 = G[V_2]$. We note that in this case, $G_2 = G'$.
\item If the tree $G'$ is not twin-free, then necessarily, the parent $x$ of $w$ has a twin $x'$ in $G'$, and $N_G(x') = N_G(x) \setminus \{w\} = \{y\}$. Thus, $d_G(x) = 2$ and the vertex $x'$ is a leaf-neighbor of $y$ in $G$. Moreover, we claim that if $x'=r$, then we are done. Indeed, in this case, our choice of $P$ as a longest path in $G$ implies that $G'$ is the path $ryx$. If now $G_w \ne P_3$, then the set of vertices of degree at least~$2$ in $G$ forms a LTD-set of $G$ of size strictly less than two-thirds the order, while if $G_w = P_3$, then $G$ is the path $P_6$, which is the $2$-corona of a tree $K_2$, and $\gLT(G) = \frac{2}{3}n$. In both cases we are done. Thus, we may assume that $x' \ne r$. We now let $V_1 = V(G_w) \cup \{x\}$, $V_2 = V(G) \setminus V_1$, and we let $G_1 = G[V_1]$ and $G_2 = G[V_2]$. We note that in this case, $G_2 = G' - x$. Our assumption that $x' \ne r$ implies that $G_2$ is a twin-free tree.
\end{itemize}

Let $D_2$ be a minimum LTD-set of $G_2$. Applying the induction hypothesis to the twin-free tree $G_2$, the set $D_2$ satisfies $|D_2| \le \frac{2}{3}|V_2|$. Further, if $|D_2| = \frac{2}{3}|V_2|$, then $G_2$ is the $2$-corona of a nontrivial tree. Let $D_1$ consist of $w$ and every child of $w$ of degree~$2$. Then, $|D_1| \le \frac{2}{3}|V_1|$ with strict inequality if $G_1$ is not the path $uvw$. We claim that $D = D_1 \cup D_2$ is a LTD-set of $G$. Since $D_1$ and $D_2$ are TD-sets of $G_1$ and $G_2$, respectively, the set $D$ is a TD-set of $G$. Every vertex of $G$ is located by $D$ except possibly for the vertex~$x$ and a leaf-neighbor of $w$ in $G$, if such a leaf-neighbor exists. If $x \in V(G_2)$, then it is located in $G_2$ and hence in $G$. If $x \in V(G_1)$, then its twin $x'$ in $G'$ is a leaf-neighbor of $y$, implying that in $G_2$ the support vertex $y \in D_2$. Thus, $x$ is located by $w$ and $y$. If $w$ has a leaf-neighbor in $G$, then such a leaf-neighbor is located by $w$ only. Therefore, $D$ is a LTD-set of $G$, and so
\begin{equation}
\gLT(G) \le |D| = |D_1| + |D_2| \le \frac{2}{3}|V_1| + \frac{2}{3}|V_2| = \frac{2}{3}n.
\label{Eq2}
\end{equation}

This establishes the desired upper bound. Suppose next that $\gLT(G) = \frac{2}{3}n$. Then we must have equality throughout the Inequality Chain~(\ref{Eq2}). In particular, $|D_1| = \frac{2}{3}|V_1|$ and $|D_2| = \frac{2}{3}|V_2|$, implying that $G_1 = P_3$ (and $G_1$ consists of the path $uvw$) and $G_2$ is the $2$-corona of a nontrivial tree, say $T_2$. Let $A$ and $B$ be the set of leaves and support vertices, respectively, in $G_2$, and let $C$ be the remaining vertices of $G_2$. We note that $C = V(T_2) = V_2 \setminus (A \cup B)$ and $|C| \ge 2$ (since $T_2$ is a nontrivial tree). If $x \in A$, then $x$ is a leaf in $G_2$ and its neighbor $y$ is a support vertex in $G_2$ and belongs to the set $B$. If $x \in B$, then $x$ is a support vertex in $G_2$ and its parent $y$ belongs to $C$. In both cases, the set $(B \cup C \cup \{v,w\}) \setminus \{y\}$ is a LTD-set of $G$ of size~$|D_1| + |D_2| - 1 = \frac{2}{3}n - 1$, a contradiction to our supposition that $\gLT(G) = \frac{2}{3}n$. Hence, $x \in C$, implying that $G$ is the $2$-corona of a nontrivial tree, namely the tree $G[C \cup \{w\}]$ obtained from $T_2$ by adding to it the vertex $w$ and the edge $wx$. This completes the proof of Claim~\ref{claimtree}.~\smallqed

\medskip
By Claim~\ref{claimtree}, we may assume that $G$ is not a tree, for otherwise the desired result follows. Hence, $G$ contains a cycle. We consider next the case when $G$ contains a triangle.

\medskip
\begin{clm}\label{c:triangle}
If $G$ contains a triangle, then $\gLT(G) \le \frac{2}{3}n$. Further, $\gLT(G) = \frac{2}{3}n$ if and only if $G$ is isomorphic to a $6$-cycle or is the $2$-corona of some connected nontrivial graph that contains no $4$-cycles but contains a triangle.
\end{clm}
\noindent
\emph{Proof of Claim~\ref{c:triangle}.} Suppose that $G$ contains a triangle $C$.
Let $G' = G - V(C)$. We build a subset $V_1$ of vertices of $G$ as follows. Let $V_0$ consist of $V(C)$ together with all vertices that belong to a component $C'$ of $G'$ isomorphic to $P_1$, $P_2$ or $P_3$. We remark that if $C'$ is a $P_1$- or $P_2$-component of $G'$, then at most one edge joins it to $C$, for otherwise there would be a $4$-cycle or a pair of twins in $G$.
Suppose that $S$ is a set of mutual twins of $G - V_0$. Since $G$ is twin-free, all but possibly one vertex in $S$ must be adjacent to a vertex of $C$. For each such set $S$ of mutual twins of $G - V_0$, we select $|S|-1$ vertices from $S$ that have a neighbor in $C$, and add these vertices to the set $V_0$ to form the set $V_1$ (possibly, $V_1 = V_0$). Let $V_2 = V(G)\setminus V_1$. Let $G_1 = G[V_1]$ and if $V_2 \ne \emptyset$, let $G_2 = G[V_2]$. We note that $G_1$ is connected, while $G_2$ may possibly be disconnected.

\begin{unnumbered}{Subclaim~8.B.1}
$G_2$ is twin-free and has no isolated vertices.
\end{unnumbered}
\emph{Proof of Subclaim~8.B.1.} We first prove that $G_2$ is twin-free. Suppose, to the contrary, that there is a pair $\{t,t'\}$ of twins in $G_2$.
By construction of $V_2$, the vertices $t$ and $t'$ are not twins in $G - V_0$, implying that there exists a vertex $v$ in $V_1 \setminus V_0$ such that $v$ is adjacent to exactly one of $t$ and $t'$, say to $t$. Let $v'$ be the twin of $v$ in $G-V_0$ that was not added to the set $V_1$ (recall that by construction, all but one vertex from a set of mutual twins in $G - V_0$ is added to the set $V_1$). But then, $v'$ is a vertex in $G_2$ that is adjacent to $t$ but not to $t'$, contradicting our supposition that $t$ and $t'$ are twins in $G_2$. Therefore, $G_2$ is twin-free. The proof that $G_2$ has no isolated vertices, again by the construction, an isolated vertex $x$ would have been a neighbor of a set of twins of $G-V_0$. But at least one twin still belongs to $G_2$, and $x$ is not isolated.~\smallqed

\medskip By Subclaim~8.B.1, $G_2$ is twin-free.  Let $D_2$ be a minimum LTD-set of $G_2$. Applying the induction hypothesis to each component of $G_2$, the set $D_2$ satisfies $|D_2| \le \frac{2}{3}|V_2|$. Further, if $|D_2| = \frac{2}{3}|V_2|$, then each component of $G_2$ is isomorphic to a $6$-cycle or is the $2$-corona of some connected nontrivial graph that contains no $4$-cycles.

We note that the graph $G_1$ could have twins. For example, this would occur if $V_1 = V(C)$, in which case $G_1$ is the $3$-cycle $C$. A more complicated possibility is if there were twins $t$ and $t'$ in $G - V_0$; then at least one of them belongs to $G_1$ and could be, in $G_1$, a twin with the vertex of some $P_1$-component of $G'$. Let us build a set $D_1\subset V_1$. As observed earlier, if $C'$ is a $P_1$- or $P_2$-component of $G'$, then at most one edge joins it to $C$. For every $P_3$-component $C'$ of $G'$, select the central vertex of $C'$ and one of its neighbors in $C'$ that is not a leaf in $G$ and add these two vertices of $C'$ to $D_1$. For every $P_2$-component $C'$ of $G'$, add to $D_1$ the unique vertex of $C'$ adjacent to a vertex of $C$, as well as its neighbor in $C$. For every $P_1$-component of $G'$ consisting of a vertex $v'$, add to $D_1$ the unique neighbor of $v'$ in $C$. For every vertex in $V_1 \setminus V_0$ that had a twin in $G - V_0$, add its neighbor in $C$ to $D_1$.
Now, if there is at most one vertex of $C$ in the resulting set $D_1$, then we augment $D_1$ so that exactly two vertices of $C$ belong to $D_1$. By construction the resulting set $D_1$ is a TD-set of $G_1$ and $|D_1| \le \frac{2}{3}|V_1|$.

\begin{unnumbered}{Subclaim~8.B.2}
$D = D_1 \cup D_2$ is a LTD-set of $G$.
\end{unnumbered}
\emph{Proof of Subclaim~8.B.2.} Since $D_1$ and $D_2$ are TD-sets of $G_1$ and $G_2$, respectively, the set $D$ is a TD-set of $G$. Suppose, to the contrary, that $D$ is not locating. Then there is a pair of vertices, $u$ and $v$, that is not located by $D$. If $(u,v) \in V_1 \times V_2$ (that is, $u \in V_1$ and $v \in V_2$), then $u$ is dominated by a vertex of $D_1$ and $v$ is dominated by a vertex of $D_2$. Hence, $u$ and $v$ must both be dominated by these two vertices. But then we have a $4$-cycle in $G$, a contradiction. Hence, $(u,v) \notin V_1 \times V_2$. Analogously, $(u,v) \notin V_2 \times V_1$. Since $D_2$ is locating in $G_2$, we note that $(u,v) \notin V_2 \times V_2$. Hence, $(u,v) \in V_1 \times V_1$; that is, both $u$ and $v$ belong to $G_1$. Moreover $u$ cannot belong to $C$, for otherwise $u$ is dominated by two vertices in $D_1\cap C$ and is located. Similarly, $v\notin C$. Analogously, $u$ and $v$ cannot belong to a $P_1$-, $P_2$- or $P_3$-component of $G'$, for otherwise it would be the only vertex in $V(G)\setminus D$ that is dominated only by its unique neighbor in $D_1$. Therefore, both $u$ and $v$ belong to $V_1 \setminus V_0$ and had a twin in $G - V_0$. Let $u'$ be the twin of $u$ in $G-V_0$ that was not added to the set $V_1$, and so $u' \in V_2$. If $u$ and $u'$ are open twins in $G - V_0$, then $u'$ is a vertex of degree~$1$ in $G$, for otherwise $u$ and $u'$ belong to a $4$-cycle. For the same reason, if $u$ and $u'$ are closed twins, then $u'$ has degree~$2$ in $G$. In both cases, $u'$ has degree~$1$ in $G_2$. The unique common neighbor of $u$ and $u'$ therefore belongs to $D_2$ in order to totally dominate the vertex $u'$ in $G_2$. Thus, $u$ is dominated by a vertex of $D_1$ and a vertex of $D_2$. Since $u$ and $v$ are not located, $v$ is also dominated by these two vertices, which implies that $u$ and $v$ belong to a common $4$-cycle of $G$, a contradiction. Therefore, $D$ is a LTD-set of $G$.~\smallqed

\medskip
By Subclaim~8.B.2, the set $D = D_1 \cup D_2$ is a LTD-set of $G$, implying that the Inequality Chain~(\ref{Eq2}) presented in the proof of Claim~\ref{claimtree} holds. This establishes the desired upper bound.

Suppose next that $\gLT(G) = \frac{2}{3}n$. Then we must have equality throughout the Inequality Chain~(\ref{Eq2}). In particular, $|D_1| = \frac{2}{3}|V_1|$ and $|D_2| = \frac{2}{3}|V_2|$.
Since $|D_1| = \frac{2}{3}|V_1|$, our construction of the set $D_1$ implies that no component of $G'$ is isomorphic to $P_1$ and that $V_1 = V_0$. Further, if $G'$ contains a $P_2$-component, then it has exactly three $P_2$-components each being joined via exactly one edge to a distinct vertex of $C$. In addition, there may be some, including the possibility of none, $P_3$-components in $G'$. Suppose that $P'$ is a $P_3$-component in $G'$ and $x$ is a vertex of $P'$ that is adjacent to a vertex of $C$. Then, $x$ is a leaf of $P'$ and is adjacent to exactly one vertex of $C$, since $G$ is twin-free and has no $4$-cycles. Suppose, further, that both leaves of $P'$ are adjacent to (distinct) vertices of $C$. Let $u$ and $v$ be two (distinct) vertices of $C$ joined to $P'$. If exactly one of $u$ and $v$ belong to $D_1$, then by our earlier observations, $G'$ contains no $P_2$-component. But then by the way in which the set $D_1$ is constructed and recalling that $G'$ contains no $P_1$-component and that $V_1 = V_0$, we would have chosen two arbitrary vertices of $C$ to add to $D_1$. Hence, we can replace the two vertices of $C$ that currently belong to $D_1$ with the two vertices $u$ and $v$. We may therefore assume that $D_1$ is chosen to contain both $u$ and $v$. With this assumption, we can replace the two vertices of $P'$ that currently belong to $D_1$ with one of the leaves of $P'$ to produce a new LTD-set of $G$ of size $|D| - 1 = \gLT(G) - 1$, a contradiction. Therefore, $P'$ is joined via exactly one edge to a vertex of $C$. Thus, there are two possible structures of the graph $G_1$, described as follows.

\begin{unnumbered}{Structure~1.}
The graph $G_1$ is obtained from the $3$-cycle $C$ by adding any number of vertex-disjoint copies of $P_3$, including the possibility of zero, and joining an end from each such added path to exactly one vertex of $C$.
\end{unnumbered}

\begin{unnumbered}{Structure~2.}
The graph $G_1$ is obtained from the $2$-corona of the $3$-cycle $C$ by adding any number of vertex-disjoint copies of $P_3$, including the possibility of zero, and joining an end from each such added path to exactly one vertex of $C$.
\end{unnumbered}

We note that if $G_1$ has the structure described in Structure~2, then $G_1$ is the $2$-corona of some connected nontrivial graph, say $H_1$, that contains the triangle $C$ and contains no $4$-cycles. Further we note that if $x \in V(H_1)$, then either $x \in V(C)$ or $x$ is the vertex of a $P_3$-component in $G'$ that is adjacent to a vertex of $C$.

\begin{unnumbered}{Subclaim~8.B.3}
If $G = G_1$,  then the graph $G$ is the $2$-corona of some connected nontrivial graph that contains the triangle $C$ and contains no $4$-cycles.
\end{unnumbered}
\emph{Proof of Subclaim~8.B.3.}   Suppose that $G = G_1$, i.e., $V_2=\emptyset$. We first show that $G$ has the structure described in Structure~2. Suppose to the contrary that $G$ has the structure described in Structure~1. Then, since $G$ is twin-free, the graph $G$ is obtained from the $3$-cycle $C$ by adding $k \ge 2$ vertex-disjoint copies of $P_3$ and joining an end from each such added path to exactly one vertex of $C$. Further, by the twin-freeness of $G$, at least two vertices of $C$ are joined to an end of an added path. Let $u$ and $v$ be two (distinct) vertices of $C$ are joined to ends of added paths $P_3$. The set of $2k$ vertices of degree~$2$ in $G$ that belong to added paths, together with the vertex~$u$, forms a LTD-set of $G$ of size~$\frac{2}{3}n - 1$, a contradiction. Therefore, $G$ has the structure described in Structure~2. Thus, the graph $G$ is the $2$-corona of some connected nontrivial graph that contains the triangle $C$ and contains no $4$-cycles.~\smallqed

\medskip
By Subclaim~8.B.3, we may assume that $G \ne G_1$, for otherwise the desired result follows. Hence, $V_2 \ne \emptyset$. Since $|D_2| = \frac{2}{3}|V_2|$, applying the inductive hypothesis to each component of $G_2$, we deduce that each component of $G_2$ is isomorphic to a $6$-cycle or is the $2$-corona of some connected nontrivial graph that contains no $4$-cycles.

\begin{unnumbered}{Subclaim~8.B.4}
No component of $G_2$ is isomorphic to a $6$-cycle.
\end{unnumbered}
\emph{Proof of Subclaim~8.B.4.}   Suppose, to the contrary, that $G_2$ contains a component $C'$ that is isomorphic to a $6$-cycle. Since $G$ is connected, there is an edge that joins a vertex $x \in V(C)$ and a vertex $y \in V(C')$. Let $C'$ be given by $y_1 y_2 \ldots y_6 y_1$, where $y = y_1$. If $G_1$ has the structure described in Structure~1, then we can choose $D_1$ to contain any two vertices of $C$. Hence we may assume that in this case, $D_1$ is chosen to contain the vertex~$x$. If $G_1$ has the structure described in Structure~2, then $V(C) \subset D_1$. In particular, $x \in D_1$. Hence, in both cases, $x \in D_1$. Replacing the four vertices of $D$ that belong to the component $C'$ with the three vertices $\{y_3,y_4,y_5\}$ produces a LTD-set of $G$ of size~$|D| - 1 = \frac{2}{3}n - 1$, a contradiction.~\smallqed

\medskip
By Subclaim~8.B.4, each component of $G_2$ is the $2$-corona of some connected nontrivial graph that contains no $4$-cycles, implying that the graph $G_2$ is the $2$-corona of some graph, say $H_2$, that contains no $4$-cycles. Moreover, since $G_2$ is twin-free, each component of $H_2$ is nontrivial. Let $A_2$ and $B_2$ be the set of leaves and support vertices, respectively, in $G_2$, and let $C_2$ be the remaining vertices of $G_2$. We note that $C_2 = V(H_2) = V_2 \setminus (A_2 \cup B_2)$.


\begin{unnumbered}{Subclaim~8.B.5}
$G_1$ has the structure described in Structure~2.
\end{unnumbered}
\emph{Proof of Subclaim~8.B.5.}   Suppose, to the contrary, that $G_1$ has the structure described in Structure~1. Then, the graph $G_1$ is obtained from the $3$-cycle $C$ by adding $k \ge 0$ vertex-disjoint copies of $P_3$ and joining an end from each such added path to exactly one vertex of $C$. Let $V(C) = \{u,v,w\}$. If at least two vertices of $C$ are joined to an end of an added path, then analogously as in the proof of Subclaim~8.B.3, we produce a LTD-set of $G$ of size~$\frac{2}{3}n - 1$, a contradiction. Hence, either $G_1 = C_3$ or $G_1$ is obtained from the $3$-cycle $C$ by adding $k \ge 1$ vertex-disjoint copies of $P_3$ and joining an end from each such added path to the same vertex of $C$, say to $u$. In both cases, both $v$ and $w$ have degree~$2$ in $G_1$. Since $G$ is twin-free, at least one of $v$ and $w$, say $v$, is adjacent to a vertex of $V_2$. If $v$ is adjacent to a vertex of $A_2 \cup B_2$, then an analogous argument as in the last paragraph of the proof of Claim~\ref{claimtree} produces a LTD-set of $G$ of size~$\frac{2}{3}n - 1$, a contradiction. Hence, the neighbors of $v$ in $V_2$ all belong to $C_2$. Analogously, the neighbors of $u$ and $w$ in $V_2$, if any exist, all belong to $C_2$.
The set of $2k$ vertices of degree~$2$ in $G$ that belong to the added $P_3$-paths in $G_1$, together with the set $B_2 \cup C_2 \cup \{v\}$, is a LTD-set of $G$ of size~$\frac{2}{3}n - 1$, a contradiction.~\smallqed

\medskip
By Subclaim~8.B.5, $G_1$ has the structure described in Structure~2, implying that $G_1$ is the $2$-corona of some connected nontrivial graph, say $H_1$, that contains the triangle $C$ and contains no $4$-cycles. Let $A_1$ and $B_1$ be the set of leaves and support vertices, respectively, in $G_1$, and let $C_1$ be the remaining vertices of $G_1$. We note that $C_1 = V(H_1) = V_1 \setminus (A_1 \cup B_2)$.

Since $G$ is connected, there is an edge in $G$ joining a vertex $x \in V_1$ and a vertex $y \in V_2$. Let $a_1b_1c_1$ be the path in $G_1$ containing $x$, where $a_1 \in A_1$, $b_1 \in B_1$ and $c_1 \in C_1$. Similarly, let $a_2b_2c_2$ be the path in $G_2$ containing $y$, where $a_2 \in A_2$, $b_2 \in B_2$ and $c_2 \in C_2$. We show that $x = c_1$. Suppose, to the contrary, that $x \in \{a_1,b_1\}$. Let $D^* = C_1 \cup C_2 \cup B_1 \cup B_2$.
If $xy=a_1a_2$, let $X = (D^* \cup \{a_1,a_2\}) \setminus \{b_1,b_2,c_1\}$.
If $xy \in \{a_1b_2,a_1c_2\}$, let $X = (D^* \cup \{a_1\}) \setminus \{b_1,c_1\}$.
If $xy=b_1a_2$, let $X = (D^* \cup \{a_2\}) \setminus \{b_2,c_2\}$.
If $xy=b_1b_2$, let $X = D^* \setminus \{c_2\}$.
If $xy=b_1c_2$, let $X = D^* \setminus \{c_1\}$.
Note that in all cases, $X$ is clearly a TD-set. To see that it is also locating, we observe that any vertex of $G_i$, $i \in [2]$, not in $X$ has a neighbor in $X \cap V(G_i)$ (to this end, also recall that $H_1$ and $H_2$ have no isolated vertices). Thus, if we had two vertices that are not located by $X$, we would have a $4$-cycle in $G$, a contradiction.
Hence, in each case the set $X$ is a LTD-set of $G$ of size~$|D| - 1 = \frac{2}{3}n - 1$, a contradiction. Therefore, $x = c_1$. Analogously, $y = c_2$. This is true for every edge $xy$ joining a vertex $x \in V_1$ and a vertex $y \in V_2$, implying that $G$ is the $2$-corona of some connected nontrivial graph that contains no $4$-cycles but contains a triangle. This completes the proof of Claim~\ref{c:triangle}.~\smallqed

\bigskip
By Claim~\ref{c:triangle}, the graph $G$ contains no triangle, for otherwise the desired result follows. Hence, the girth of $G$ is at least~$5$. Let $C \colon u_0u_1 \ldots u_{k-1} u_0$ ($k\geq 5$) be a smallest cycle in $G$. Let $G' = G - V(C)$.
We build a subset $V_1$ of vertices of $G$ as follows (similarly to the proof of Claim~\ref{c:triangle}). Let $V_0$ consist of $V(C)$ together with all vertices that belong to a component of $G'$ isomorphic to $P_1$, $P_2$ or $P_3$. Since $G$ is twin-free and has girth at least~$5$, we note that $G[V_0]$ is twin-free.
Suppose that $S$ is a set of mutual twins of $G - V_0$. Since $G$ is twin-free, all but possibly one vertex in $S$ must be adjacent to a vertex of $C$. For each such set $S$ of mutual twins of $G - V_0$, we select $|S|-1$ vertices from $S$ that have a neighbor in $C$, and add these vertices to the set $V_0$ to form the set $V_1$ (possibly, $V_1 = V_0$). Let $T = V_1 \setminus V_0$. We note that since $G$ has girth at least~$5$, the vertices in each set $S$ of mutual twins of $G - V_0$ are open twins, and have degree~$1$ in $G - V_0$ (if they were closed twins, they could not have a common neighbor since $G$ has girth at least~$5$, but then they would form a $P_2$-component of $G'$). Moreover they can have at most one neighbor in $V_0$, for otherwise they would have two or more neighbors in $V(C)$, but this would create a shorter cycle than $C$, contradicting its minimality. Hence, every vertex in $T$ has exactly one neighbor in $V_0$ (more precisely, in $V(C)$). Let $V_2 = V(G)\setminus V_1$. Let $G_1 = G[V_1]$ and if $V_2 \ne \emptyset$, let $G_2 = G[V_2]$.

\medskip
\begin{clm}\label{c:G1}
If $G = G_1$, then $\gLT(G) \le \frac{2}{3}n$. Further, $\gLT(G) = \frac{2}{3}n$ if and only if $G$ is isomorphic to a $6$-cycle or is the $2$-corona of the cycle $C$.
\end{clm}
\noindent
\emph{Proof of Claim~\ref{c:G1}.} Suppose that $G = G_1$. If $T \ne \emptyset$, then this would imply that $V_2 \ne \emptyset$, contradicting our supposition that $V(G) = V_1$. Hence, $T = \emptyset$, and so $V_1 = V_0$. Thus, either $G$ is the $k$-cycle $C$ or $V(G) \ne V(C)$ and every component in $G' = G - V(C)$ is isomorphic to $P_1$, $P_2$ or $P_3$. Suppose that $G = C$. Then, $n = k$. If $k = 5$, then $G = C_5$ and $\gLT(G) = 3 < \frac{2}{3}n$. If $k = 6$, then $G = C_6$ and $\gLT(G) = \frac{2}{3}n$. If $G = C$ and $k > 6$, then, as observed in~\cite{hhh06}, $\gLT(G) = \gt(G) = \lfloor n/2 \rfloor + \lceil n/4 \rceil - \lfloor n/4 \rfloor \le \frac{1}{2}n + 1 < \frac{2}{3}n$. Hence we may assume that $G \ne C$, for otherwise the desired result follows. As observed earlier, every component of $G'$ is isomorphic to $P_1$, $P_2$ or $P_3$. Among all components of $G'$, let $P'$ be chosen so that its order is maximum. We now consider the graph $F = G - V(P')$. Clearly, $F$ is twin-free, since $G$ is twin-free and removing $P'$ from $G$ cannot create any twins. Applying the inductive hypothesis to the graph $F$, $\gLT(F) \le \frac{2}{3}|V(F)|$. Further, $\gLT(F) = \frac{2}{3}|V(F)|$ if and only if $F$ is isomorphic to a $6$-cycle, $C_6$, or is the $2$-corona of some connected nontrivial graph that contains no $4$-cycles.

\begin{unnumbered}{Subclaim~8.C.1}
If $\gLT(F)<\frac{2}{3}|V(F)|$, then the desired result of Claim~\ref{c:G1} holds.
\end{unnumbered}
\emph{Proof of Subclaim~8.C.1.}   Suppose that $\gLT(F)<\frac{2}{3}|V(F)|$. If $P'=P_3$, consider a minimum LTD-set $D_F$ of $F$, and note that $D_F$ together with the two vertices of $P'$ that have degree at least~$2$ in $G$, forms a LTD-set of $G$ of size strictly less that $\frac{2}{3}n$. Hence, we may assume that $P'$ is isomorphic to $P_1$ or $P_2$. By our choice of $P'$, this implies that every component of $G'$ is isomorphic to $P_1$ or $P_2$. We now construct a set $Q$ with $V(P')\subset Q$. Renaming vertices of $C$, if necessary, we may assume that $u_1$ is the vertex of $C$ adjacent to a vertex of $P'$. We initially define $Q$ to contain both $u_1$ and $u_2$, as well as all vertices that belong to a $P_1$- or $P_2$-component of $G-\{u_1,u_2\}$.
If $u_3$ has degree~$2$ in $G$ and $u_4$ has a leaf-neighbor in $G$, say $u_4'$, then $u_3$ and $u_4'$ are (open) twins in $G - Q$. In this case, we add the vertex $u_3$ to the set $Q$.
Analogously, if $u_0$ has degree~$2$ in $G$ and $u_{k-1}$ has a leaf-neighbor in $G$,  then we add the vertex $u_0$ to the set $Q$.
By construction, the resulting graph $G - Q$ is twin-free, unless we have the special case when $k = 5$, both $u_0$ and $u_3$ have degree~$2$ in $G$, and $u_4$ has degree~$3$ in $G$ with a leaf-neighbor in $G$. In this case, graph $G$ is determined and the set $\{u_0,u_1,u_2,u_4\}$ together with the vertices of every $P_2$-component in $G'$ that have a neighbor in $V(C)$ forms a LTD-set of $G$ of size strictly less that $\frac{2}{3}n$. Hence, we may assume that the graph $F' = G - Q$ is twin-free.

Applying the inductive hypothesis to the graph $F'$ there exists a LTD-set, $D_F'$, of $F'$ of size at most $\frac{2}{3}|V(F')|$. Although $G[Q]$ is not necessarily twin-free, by similar arguments as before we can easily choose a set $D_Q$ of size at most $\frac{2}{3}|Q|$ such that $D_F'\cup D_Q$ is a LTD-set of $G$ of size at most $\frac{2}{3}n$. Moreover, if $|D_F'\cup D_Q|=\frac{2}{3}n$, then $F'$ must be either the $2$-corona of the path $G[V(C)\setminus Q]$, or $F' = P_6$.
Furthermore, $|Q| = 6$ and $G[Q]$ is either a $P_6$, a $P_4$ with an additional leaf attached to each central vertex, or a $P_5$ with an additional leaf forming a twin with another leaf. If $F'=P_6$ or $G[Q]\neq P_6$, we can readily find a LTD-set  of $G$ strictly smaller than $\frac{2}{3}n$. Otherwise, $G$ is the $2$-corona of $C$, and we are done. This completes the proof of Subclaim~8.C.1.~\smallqed

\medskip
By Subclaim~8.C.1, we may assume that $\gLT(F)=\frac{2}{3}|V(F)|$, for otherwise the desired result follows. If $F = C_6$, then $\gLT(G) < \frac{2}{3}n$, irrespective of whether $P'$ is isomorphic to $P_1$, $P_2$ or $P_3$. Hence, we may assume that $F \ne C_6$, for otherwise the desired result follows. Thus, $F$ is the $2$-corona of some connected nontrivial graph, say $F'$, that contains no $4$-cycles. Let $A_F$ and $B_F$ be the set of leaves and support vertices, respectively, in $F$, and let $C_F$ be the remaining vertices of $F$. Thus, $F' = F[C_F]$. If $P'$ is not isomorphic to $P_3$, or if $P'$ is isomorphic to $P_3$ and contains a vertex adjacent to $A_F$ or $B_F$, then it is a simple exercise to see that $\gLT(G) < \frac{2}{3}n$. Further, if $P'$ is isomorphic to $P_3$ and contains two or more vertices adjacent to vertices of $C_F$, then $\gLT(G) < \frac{2}{3}n$. If $P'$ is isomorphic to $P_3$ and contains exactly one vertex adjacent to vertices of $C_F$, then $\gLT(G) = \frac{2}{3}n$ and $G$ is the $2$-corona of some connected nontrivial graph that contains no $4$-cycles. This completes the proof of Claim~\ref{c:G1}.~\smallqed

\medskip
By Claim~\ref{c:G1}, we may assume that $G \ne G_1$, i.e., $V_2 \ne \emptyset$. An identical proof as in the proof of Subclaim~8.B.1 shows that $G_2$ is twin-free. Let $D_2$ be a minimum LTD-set of $G_2$. Applying the induction hypothesis to each component of $G_2$, the set $D_2$ satisfies $|D_2| \le \frac{2}{3}|V_2|$. Further, if $|D_2| = \frac{2}{3}|V_2|$, then each component of $G_2$ is isomorphic to a $6$-cycle or is the $2$-corona of some connected nontrivial graph that contains no $4$-cycles.

Recall that $G[V_0]$ is twin-free. We now build sets $V_1'$ and $T'$ such that $V_0 \subseteq V_1' \subseteq V_1 = V_0 \cup T$ and $T' \subseteq T$, as follows. Initially, we let $V_1' = V_0$ and $T' = T$. We consider the vertices of $T$ sequentially. Let $t$ be a vertex in $T$, and recall that $t$ has exactly one neighbor, say $u_t$, in $V_0$, and such a neighbor belongs to $V(C)$. If $u_t$ has no leaf-neighbor in $G[V_1']$, we add $t$ to $V_1'$ and remove $t$ from $T'$. We iterate this process until all vertices of $T$ have been considered. Let $G_1'$ be the resulting graph $G[V_1']$.
This process yields a new partition of $V(G)$ into sets $V_2$, $V_1'$ and $T'$. Since $G[V_0]$ is twin-free, by construction of the set $V_1'$, the graph $G_1'$ is also twin-free. Since $V_2 \ne \emptyset$, the order of $G_1'$ is less than~$n$ and we can therefore apply the induction hypothesis to the connected twin-free graph $G_1'$. Let $D_1'$ be a minimum LTD-set of $G_1'$. By the induction hypothesis, the set $D_1'$ satisfies $|D_1'| \le \frac{2}{3}|V_1'| \le \frac{2}{3}|V_1|$. Further, if $|D_1'| = \frac{2}{3}|V_1'|$, then $G_1'$ is isomorphic to a $6$-cycle or is the $2$-corona of some connected nontrivial graph that contains no $4$-cycles.

We claim that $D = D_1' \cup D_2$ is a LTD-set of $G$. By the construction of the set $T'$, for each vertex $t$ of $T'$, there is a twin, say $t'$, of $t$ in $G - V_0$ that belongs to $V_2$ and has degree~$1$ in $G_2$. The common neighbor of $t$ and $t'$ in $V_2$ must belong to $D_2$. Further, since $t$ has not been removed from $T'$ during the construction of $T'$, the vertex $t$ has a neighbor $u_t$ in $V(C)$ which has a leaf-neighbor in $G_1'$, implying that the vertex $u_t$ belongs to $D_1'$. Hence, $t$ is dominated by two vertices of $D_1' \cup D_2$ and is therefore located by $D$, for otherwise we would have a $4$-cycle in $G$. Thus, every vertex of $T'$ is located by $D$. Since $D_1'$ and $D_2$ are TD-sets of $G_1'$ and $G_2$, respectively, and since every vertex in $T'$ is dominated by $D$, the set $D$ is a TD-set of $G$. Suppose, to the contrary, that $D$ is not locating. Then there is a pair of vertices, $u$ and $v$, that is not located by $D$. As observed earlier, neither $u$ nor $v$ belong to $T'$. Since $D_2$ is locating in $G_2$, we note that $(u,v) \notin V_2 \times V_2$. Analogously, since $D_1'$ is locating in $G_1'$, we note that $(u,v) \notin V_1' \times V_1'$. If $(u,v) \in V_1' \times V_2$, then $u$ is dominated by a vertex of $D_1'$ and $v$ is dominated by a vertex of $D_2$. Hence, $u$ and $v$ must both be dominated by these two vertices. But then these four vertices would form a $4$-cycle, a contradiction. Hence, $(u,v) \notin V_1' \times V_2$. Analogously, $(u,v) \notin V_2 \times V_1'$. This contradicts our supposition that $u$ and $v$ are not located by $D$. Therefore, $D$ is a LTD-set of $G$, and so
\begin{equation}
\gLT(G) \le |D| = |D_1'| + |D_2| \le \frac{2}{3}|V_1'| + \frac{2}{3}|V_2| \le  \frac{2}{3}|V_1| + \frac{2}{3}|V_2| = \frac{2}{3}n.
\label{Eq3}
\end{equation}

This establishes the desired upper bound. Suppose next that $\gLT(G) = \frac{2}{3}n$. Then we must have equality throughout the Inequality Chain~(\ref{Eq3}). In particular, $|D_1'| = \frac{2}{3}|V_1'| = \frac{2}{3}|V_1|$ and $|D_2| = \frac{2}{3}|V_2|$. This in turn implies that $T' = \emptyset$. Using an analogous proof as in the proof when equality holds in the Inequality Chain~(\ref{Eq2}) in the proof of Claim~\ref{c:triangle}, the graph $G$ can be shown to be the $2$-corona of some connected nontrivial graph that contains no $4$-cycles. Since the proof is very similar, we omit the details. This completes the proof of Theorem~\ref{t:no4cycle}.
\end{proof}

\section{Graphs with given minimum degree}\label{sec:mindegree}

We now discuss the special case of graphs of given minimum degree.

\subsection{Minimum degree two}

If we forbid a certain set of six graphs (each of them of order at most~$10$), then it is known (see~\cite{H00}) that every connected graph $G$ of order~$n$ with $\delta(G) \ge 2$ satisfies $\gt(G) \le 4n/7$. However, for graphs with minimum degree~$2$, the location-total domination number can be much larger than the total domination number. For example, let $G$ be the graph obtained by taking the disjoint union of $k \ge 2$ $5$-cycles, adding a new vertex~$v$ and joining $v$ with an edge to exactly one vertex from each $5$-cycle. The resulting twin-free graph $G$ has order~$n = 5k+1$, minimum degree~$\delta(G) = 2$ and satisfies $\gLT(G) = 3k = \frac{3}{5}(n-1)$ and $\gt(G) = 2(k+1) = \frac{2}{5}(n-1) + 2$.

We believe that Conjecture~\ref{conj} can be strengthened for graphs with minimum degree at least~$2$ and pose the following question.

\begin{quest}
Is it true that every twin-free graph with order $n$, no isolated vertices and minimum degree~$2$ satisfies $\gLT(G)\le \frac{3n}{5}$?
\label{Q1}
\end{quest}

If Question~\ref{Q1} is true, then the bound is asymptotically tight by the examples given earlier.

\subsection{Large minimum degree}

The following is an upper bound on $\gt(G)$ according to the minimum degree $\delta$ of $G$.

\begin{thm}[Henning,Yeo~\cite{HY07}] \label{rhm:delta}
If $G$ is a graph with minimum degree~$\delta \ge 1$ and order~$n$,
then
\[
\gt(G) \le \left( \frac{ 1 + \ln \delta } { \delta } \right)
n.
\]
\end{thm}

Using Observation~\ref{ob:D=T+L}, we obtain the following corollary of the results in~\cite{Heia,conjpaper} and Theorem~\ref{rhm:delta}.

\begin{cor}\label{thm:mindeg}
Let $G$ be a twin-free graph of minimum degree~$\delta\ge  1$. We have
\[
\gLT(G) \le  \left(\frac{2}{3}+\frac{1+\ln\delta}{\delta}\right)n.
\]
Moreover, if $G$ is a bipartite, co-bipartite or split graph, then
\[
\gLT(G)\le \left(\frac{1}{2}+\frac{1+\ln\delta}{\delta}\right)n.
\]
If Conjecture~\ref{conj-LD} holds, we always have $\gLT(G)\le  \left(\frac{1}{2}+\frac{1+\ln\delta}{\delta}\right)n$.
\end{cor}

It follows from Corollary~\ref{thm:mindeg} that Conjecture~\ref{conj} asymptotically holds for large minimum degree, in the sense that $\lim_{\delta\rightarrow\infty}\left(\frac{2}{3}+\frac{1+\ln\delta}{\delta}\right)=\frac{2}{3}$. Moreover, Conjecture~\ref{conj} holds for bipartite, co-bipartite, and split graphs with minimum degree $\delta\ge  26$. Finally, if Conjecture~\ref{conj-LD} holds, then Conjecture~\ref{conj} holds whenever $\delta\ge  26$.

\section{Conclusion}\label{sec:conclu}

A classic result in total domination theory in graphs is that every connected graph of order $n \ge 3$ has a total dominating set of size at most $\frac{2}{3}n$. In this paper, we conjecture that every twin-free graph of order~$n$ with no isolated vertex has a locating-total dominating set of size at most $\frac{2}{3}n$ and we prove our conjecture for graphs with no $4$-cycles. We also prove that our conjecture, namely Conjecture~\ref{conj}, holds asymptotically for large minimum degree.
Since Conjecture~\ref{conj-LD} was proved for bipartite graphs~\cite{conjpaper} and cubic graphs~\cite{cubic}, can we prove Conjecture~\ref{conj} for these classes as well?

\end{document}